\documentclass{ifacconf}

\usepackage{graphicx}      
\usepackage{natbib}        
\usepackage{amsmath} 
\usepackage{amssymb}  
\usepackage{amsfonts, latexsym, amscd, enumerate, bbm, nicefrac, mathrsfs}
\usepackage{algcompatible}
\usepackage{algorithm}
\usepackage{textcomp}
\usepackage[utf8]{inputenc}
\usepackage{float}
\usepackage{tikz,pgfplots}
\usetikzlibrary{patterns}
\usetikzlibrary{quotes,angles}
\usepackage{url}
\usepackage{courier}
\usepackage{subfigure}
\usepackage{booktabs}

\usetikzlibrary{arrows.meta}

\newtheorem{definition}{Definition}
\newtheorem{lemma}{Lemma}

\newtheorem{theorem}{Theorem}
\newtheorem{assumption}{Assumption}
\newtheorem{remark}{Remark}

\pgfplotsset{compat=1.17}

\makeatletter
\DeclareRobustCommand{\qed}{%
  \ifmmode 
  \else \leavevmode\unskip\penalty9999 \hbox{}\nobreak\hfill
  \fi
  \quad\hbox{\qedsymbol}}
\newcommand{\openbox}{\leavevmode
  \hbox to.77778em{%
  \hfil\vrule
  \vbox to.675em{\hrule width.6em\vfil\hrule}%
  \vrule\hfil}}
\newcommand{\qedsymbol}{\openbox}
\newenvironment{proof}[1][\proofname]{\par
  \normalfont
  \topsep6\p@\@plus6\p@ \trivlist
  \item[\hskip\labelsep\itshape
    #1.]\ignorespaces
}{%
  \qed\endtrivlist
}
\newcommand{\proofname}{Proof}
\makeatother

\begin{document}
\begin{frontmatter}

\title{Anderson Accelerated Feasible Sequential Linear Programming\thanksref{footnoteinfo}} 

\thanks[footnoteinfo]{This work has been carried out within the framework of Flanders
Make SBO DIRAC: DIRAC - Deterministic and Inexpensive Realizations of Advanced
Control.}

\author[First]{David Kiessling} 
\author[Second]{Pieter Pas}
\author[First]{Alejandro Astudillo}
\author[Second]{Panagiotis Patrinos}
\author[First]{Jan Swevers}

\address[First]{MECO~Research~Team,~Dept.~of~Mechanical~Engineering,~KU~Leuven.\\
Flanders Make - DMMS-M, Leuven, Belgium.
    (e-mail: \{david.kiessling, alejandro.astudillovigoya, jan.swevers\}@kuleuven.be).}
\address[Second]{Department of Electrical Engineering (ESAT),
STADIUS Center for Dynamical Systems, Signal Processing and Data
Analytics, KU Leuven. (e-mail: \{pieter.pas, panos.patrinos\}@kuleuven.be)}

\begin{abstract}                
This paper proposes an accelerated version of Feasible Sequential Linear Programming (FSLP): the AA($d$)-FSLP algorithm. FSLP preserves feasibility in all intermediate iterates by means of an iterative update strategy which is based on repeated evaluation of zero-order information. This technique was successfully applied to techniques such as Model Predictive Control and Moving Horizon Estimation, but it can exhibit slow convergence. Moreover, keeping all iterates feasible in FSLP entails a large number of additional constraint evaluations. In this paper, Anderson Acceleration (AA($d$)) is applied to the zero-order update strategy improving the convergence rate and therefore decreasing the number of constraint evaluations in the inner iterative procedure of the FSLP algorithm. AA($d$) achieves an improved contraction rate in the inner iterations, with proven local linear convergence. In addition, it is observed that due to the improved zero-order update strategy, AA($d$)-FSLP takes larger steps to find an optimal solution, yielding faster overall convergence. The performance of AA($d$)-FSLP is examined for a time-optimal point-to-point motion problem of a parallel SCARA robot. The reduction of the number of constraint evaluations and overall iterations compared to FSLP is successfully demonstrated.  
\end{abstract}


\end{frontmatter}

\section{Introduction}

Optimization algorithms based on zero-order information are a popular choice in real-time optimization to keep intermediate iterates feasible. These iterative strategies are based on Newton-type optimization methods that keep derivative information fixed at a given point. Only the repeated evaluation of zero-order information, i.e., function evaluations, is required. First introduced in \cite{Bock2007} to reduce the computational effort in Sequential Quadratic Programming (SQP) methods, zero-order methods were successfully used in many different applications such as Moving Horizon Estimation (MHE) \citep{Baumgaertner2019}, Iterative Learning Control (ILC) \citep{Baumgaertner2020a}, and Model Predictive Control (MPC) \citep{Zanelli_2019}. A feasible SQP method was introduced in \cite{Zanelli2021}. Exploiting the so-called fully determined case, a Feasible Sequential Linear Programming (FSLP) algorithm was introduced in \cite{Kiessling2022}. An iterative zero-order update strategy, called feasibility iterations or inner iterations, enables feasible iterates. Feasible optimization algorithms are particularly interesting for real-time applications where execution time is limited. Keeping all iterates feasible enables a solver to terminate at suboptimal, feasible points, which is beneficial for time-critical approaches such as MPC, where a timely feasible solution is often prioritized over full convergence of the solution.\\
In general, the FSLP algorithm admits only local linear convergence in the outer iterations. Furthermore, an efficient subproblem solver is needed to solve the arising LPs in the inner iterations. Moreover, the feasibility iterations can suffer from slow convergence, which results in many additional constraint evaluations. \\
In this paper, we propose an Anderson Acceleration (AA) strategy with memory $d \in \mathbb{N}$ of previous iterates, which we call AA($d$), in order to improve on the latter drawback, i.e., the number of feasibility iterations. Introduced in the context of integral equations in \cite{Anderson1965}, AA is a widely used method to improve the convergence rate of fixed point iterations. The step update is computed by an affine combination of previous iterates and steps. Since Newton-type methods can be interpreted as fixed point iterations, AA is applicable to zero-order based methods. The local convergence of AA($d$) is proven and the convergence improvement in the inner iterations is successfully demonstrated on a time-optimal point-to-point (P2P) motion problem of a SCARA robot.\\
This paper is structured as follows. Section \ref{section_fslp} introduces the FSLP algorithm and describes the feasibility iterations. Section \ref{section_anderson} discusses the application of AA to the feasibility iterations procedure. Section \ref{section_example} presents a simulation example of a time-optimal point-to-point motion problem of a SCARA robot. Section \ref{section_conclusion} closes the paper with concluding remarks and ongoing developments.

\section{Feasible Sequential Linear Programming}
\label{section_fslp}
First, we introduce the notation and formulate the general optimization problem that will be solved in this paper. Next, we briefly discuss the FSLP algorithm with a focus on the feasibility iteration procedure.

\subsection{Notation \& Problem Formulation}
The optimization problem structure is the same as introduced in \cite{Kiessling2022}. Let $w \in \mathbb{R}^{n_w}$, then the general Nonlinear Program (NLP) is defined by
\begin{equation}
\label{eq:nlp}
\begin{split}
\min_{w \in \mathbb{R}^{n_w}} \: c^{\top}w\quad\mathrm{s.t.} \:\: Cw + g(P_yw)=0,\:\: Aw+b\leq 0.
\end{split}
\end{equation}
Here, $c\in\mathbb{R}^{n_w}$, $b\in\mathbb{R}^{n_b}$, $A\in\mathbb{R}^{n_b\times n_w}$ with full row rank $n_b$  and let $g\colon\mathbb{R}^{n_y}\to\mathbb{R}^{n_g}$ with $C\in\mathbb{R}^{n_g\times n_w}$. Further assume that $g$ is twice continuously differentiable. The projection matrix $P_y\in\mathbb{R}^{n_y\times n_w}$ is a sparse matrix, composed of rows with each a single 1-entry, that selects variables of $w$ that enter into the constraints nonlinearly. General nonlinear objective functions and inequality constraints can be cast in the form \eqref{eq:nlp} by introducing slack variables. Within this paper, iteration indices will always be represented by subscript letters, e.g., $w_1\in\mathbb{R}^{n_w}$ whereas indices of vector entries are always specified by superscript letters, e.g., $w^1\in\mathbb{R}$. The measure of infeasibility is defined by $h(w) := \Vert Cw + g(P_y w) \Vert_{\infty} + \Vert [Aw+b]^+ \Vert_{\infty}$, where $[Aw+b]^+ := [\max\{A^iw+b^i,\,0\}]_{i=1}^{n_b}$. The feasible set is denoted by $\mathcal{F} := \{w\in\mathbb{R}^{n_w}\,|\, Cw + g(P_y w)=0,\,Aw+b\leq 0\}$. Let the set of active constraints at $w\in\mathcal{F}$ be defined by $\mathcal{A}(w):= \{i\in\{1,\ldots,\,n_b\}\:|\:A^i w+b^i=0\}$. The inner product of $\mathbb{R}^{n_w}$ is defined by $\langle \cdot \; , \; \cdot \rangle : \mathbb{R}^{n_w} \times \mathbb{R}^{n_w} \rightarrow \mathbb{R}$. 
The closed ball around $\hat{w}$ with radius $\gamma>0$ is denoted by $\mathcal{B}(\hat{w},\,\gamma):=\{w\in\mathbb{R}^{n_w}\: |\:\Vert\hat{w}-w\Vert_2\leq\gamma \}$. For a given $\hat{\Delta} > 0$ and a $\hat{w}\in\mathbb{R}^{n_w}$, we define the set $\mathcal{B}_{\infty}(\hat{w},\,\hat{\Delta}) := \{w\in\mathbb{R}^{n_w}\, |\, \Vert w-\hat{w}\Vert_{\infty} \leq \hat{\Delta} \}$. The projection of a point $w\in\mathbb{R}^{n_w}$ onto a nonempty closed convex set $C\subset\mathbb{R}^{n_w}$, $\Pi_{C}(w)$, is defined by $\Pi_{C}(w) = \mathrm{argmin}_{y\in C} \Vert y-w \Vert_{2}$. Unless otherwise specified, $\Vert \cdot \Vert$ denotes the Euclidean norm. 

\subsection{Feasible Sequential Linear Programming Algorithm}
A full description of the FSLP algorithm is given by \cite{Kiessling2022}. The particularity of FSLP is that all iterates are kept feasible by an efficient iterative procedure which is described in the subsequent section. Additionally, a trust-region in combination with a merit function is used for global convergence.\\ 
The algorithm needs to be initialized at a feasible point. Let $\hat{w}\in\mathcal{F}$ be the point of linearization, then the algorithm solves a sequence of linear programs (LPs) of the following form:
\begin{subequations}
\label{eq:LP}
\begin{align}
\min_{w \in \mathbb{R}^{n_w}} \quad&c^{\top} w\\
\mathrm{s.t.}\quad &Cw + \nabla g(P_y\hat{w})^{\top}P_y (w-\hat{w}) = 0 \label{FP-QPequal},\\
\quad &Aw +b \leq 0\label{FP-QPinequal},\\
\quad &||P_y(w-\hat{w})||_{\infty} \leq \Delta.
\end{align}
\end{subequations}
In \eqref{eq:LP}, the projection matrix $P_y$, which was defined in \eqref{eq:nlp}, defines the optimization variables that are considered inside the trust-region. Variables appearing only in linear constraints do not need to be controlled by the trust-region and can therefore be excluded by $P_y$. The solution of \eqref{eq:LP} is denoted by $\bar{w}\in\mathbb{R}^{n_w}$ and is projected onto the feasible set. The resulting feasible iterate is denoted by $\tilde{w}\in\mathbb{R}^{n_w}$. For guaranteeing global convergence of FSLP, it needs to be ensured that the \textit{projection ratio condition}, i.e., $\Vert \bar{w}-\tilde{w}\Vert / \Vert \bar{w}-\hat{w}\Vert < 1/2$, is satisfied.
Since every iterate remains feasible, the merit function is chosen as the objective function. As in standard trust-region algorithms, the ratio of actual to predicted reduction decides upon step acceptance or rejection. As termination criterion, we use the linear model $m(\bar{w};\,\hat{w}) := c^{\top} (\bar{w}-\hat{w})$. If the model does not decrease through a new LP iterate $\bar{w}$, an optimal point was found and the FSLP algorithm is terminated. In the following, the projection onto the feasible set will be referred to as inner iterations. One iteration of FSLP including the feasibility projection will be denoted as an outer iteration.

\subsection{Inner Iterations}
This section introduces the inner feasibility iterations of FSLP which calculate a feasible iterate $\tilde{w}\in\mathcal{F}$ from the LP solution $\bar{w}$. 
Let $\hat{w}\in\mathcal{F}$ be the outer iterate and $w_l\in\mathbb{R}^{n_w}$ be an inner iterate for a given inner iterate counter $l\in\mathbb{N}$. The Jacobian of $g$ is fixed at $P_y\hat{w}$, i.e., $G^{\top} := \nabla g(P_y\hat{w})^{\top}P_y$. We define
\begin{align}
\label{eq:delta}
\delta(w_l,\,\hat{w}) := g(P_y w_l)-g(P_y \hat{w}) - G^{\top}(w_l-\hat{w}).
\end{align}
Introducing the notation $\delta_l := \delta(w_l,\,\hat{w})$, the parametric linear program $\mathrm{PLP}(w_l;\,\hat{w},\Delta)$ is defined as
\begin{subequations}
\label{PLP}
\begin{align}
\min_{w \in \mathbb{R}^{n_w}} \quad&c^{\top}w\\
\mathrm{s.t.}\quad &\delta_l +Cw + G^{\top} (w-\hat{w}) = 0,\label{eq:ZOnleq}\\
\quad &Aw +b \leq 0,\\
\quad &||P_y(w-\hat{w})||_{\infty} \leq \Delta.
\end{align}
\end{subequations}
Its solution will be denoted by $w_{\mathrm{PLP}}^*(w_l;\,\hat{w},\Delta)$. We note that for $w_l \leftarrow \hat{w}$, we obtain LP \eqref{eq:LP} and for $w_l \leftarrow \bar{w}$, we obtain a standard second-order correction problem as defined by \cite{Conn2000}. The algorithm is based on the iterative solution of PLPs. In every iteration, the nonlinear constraints are re-evaluated at $w_l$ in the term $\delta_l$ and another PLP is solved. The solution of this problem is denoted by $w_{l+1}$, i.e., $w_{l+1}=w_{\mathrm{PLP}}^*(w_l;\,\hat{w},\Delta)$, and the procedure is repeated. Algorithm \ref{alg:feas_iter} presents this feasibility improvement strategy in detail.\\
\begin{algorithm}[thpb]
\caption{Inner Iterations}
\label{alg:feas_iter}
\begin{algorithmic}[1]
\REQUIRE $\hat{w}\in\mathcal{F}$, fixed Jacobian $G^{\top}=\nabla g(\hat{w})^{\top}P_y$. Let $\bar{w}$ be the solution of \eqref{eq:LP} at $\hat{w}$. $\Delta>0$, $\sigma_{\mathrm{inner}}\in (0,\,10^{-5})$;
\ENSURE $\tilde{w}$
\STATE $w_0\leftarrow \bar{w}$
\FOR{$l=0,\,1,\,2, \ldots$}
\IF{$h(w_l) \leq \sigma_{\mathrm{inner}}$ and $\Vert \bar{w}-w_l\Vert / \Vert \bar{w}-\hat{w}\Vert < 1/2$}
\STATE $\tilde{w}\leftarrow w_l$
\STATE \textbf{STOP}
\ENDIF
\STATE Solve $\mathrm{PLP}(w_l,\,\hat{w},\,\Delta)$
\IF{iterates $w_l$ are diverging}
\STATE \textbf{STOP}
\ENDIF
\STATE $w_{l+1}\leftarrow w_{\mathrm{PLP}}^*(w_l;\,\hat{w},\Delta)$
\ENDFOR
\end{algorithmic}
\end{algorithm}
The feasibility iterations are repeated until convergence towards a feasible point of \eqref{eq:nlp} is achieved. If the iterates are diverging the algorithm is terminated, then the inner algorithm returns to the outer FSLP algorithm, and the trust-region radius is decreased. The trust-region in \eqref{eq:LP} is used to ensure global convergence in the outer iterations and local convergence in the inner iterations.\\
Particularly, in every iteration of Algorithm \ref{alg:feas_iter}, only the constraints are re-evaluated. This is advantageous in applications where the evaluation of first- and second-order derivatives is expensive.\\
The main contribution of \cite{Kiessling2022} is to show that the convergence of Algorithm \ref{alg:feas_iter} is depending on the size of the trust-region. For completeness of presentation, this result is stated here. The main assumption for local convergence is known as strong regularity in the context of generalized equations \citep{Izmailov2014}. 
\begin{assumption}[Strong regularity]
\label{as:invertiblePLP}
For all $\hat{w}\in\mathcal{F}$ exist $L_1,\, L_2,\,\bar{\Delta}>0$ such that for all $w_1,\,w_2\in\mathcal{B}_{\infty}(\hat{w},\,\Delta)$ with $\Delta\leq\bar{\Delta}$, it holds for all $\delta_1,\,\delta_2$ defined in \eqref{eq:delta} and $\delta_1,\,\delta_2\in \mathcal{B}(0,\,L_2\,\Delta)$ that
$\Vert w_{\mathrm{PLP}}^*(w_1;\,\hat{w},\Delta)-w_{\mathrm{PLP}}^*(w_2;\,\hat{w},\Delta)\Vert \leq L_1 \Vert \delta_1-\delta_2\Vert.$
\end{assumption}
\begin{theorem}[Local linear convergence proportional to $\Delta$]
\label{theo:localContractionJosephyNewton}
Let $\hat{w}\in\mathcal{F}$ and let Assumption \ref{as:invertiblePLP} hold, then there exist $\Delta_{\mathrm{max},\,1},\,L>0$ such that for all $\Delta\leq\Delta_{\mathrm{max},\,1}$ it holds that $L\Delta \leq 1$ and the iterates $\{w_l\}_{l\in\mathbb{N}}$ of Algorithm \ref{alg:feas_iter} converge to a point $w^*\in\mathcal{F}$. The contraction rate is proportional to $\Delta$, i.e.,
$
\Vert w_{l+1}-w^*\Vert \leq L \Delta \Vert w_{l}-w^*\Vert.
$
\end{theorem}
\begin{remark}
    Starting from local convergence of the inner iterates, the projection ratio condition that is needed for global convergence of FSLP can be proven \citep{Kiessling2022}.
\end{remark}
A drawback of Algorithm \ref{alg:feas_iter} is that it can take many iterations to find the optimal, feasible solution. This slows down the solution process since the constraints need to be re-evaluated and many additional LPs need to be solved. In this paper, an Anderson acceleration method is applied to the inner iterations in order to overcome this drawback by improving the convergence rate. 

\section{Anderson Acceleration}
\label{section_anderson}
This section describes the main contribution of this paper, an acceleration strategy to improve the convergence rate of the inner iterations.\\ 
Algorithm \ref{alg:feas_iter} can be interpreted as a fixed point iteration with fixed point operator $w_{\mathrm{PLP}}^*(\hat{w},\,\Delta) \colon \mathbb{R}^{n_w} \to \mathbb{R}^{n_w}$. Then, the fixed point is $\tilde{w} = w^*$. A popular method to accelerate the convergence of fixed point iterations is the so-called \textit{Anderson Acceleration} method which was introduced by \cite{Anderson1965} to solve integral equations. A memory $d\in\mathbb{N}$ of previous iterates is used to improve the convergence rate of the fixed point iteration.
The acceleration algorithm with a given memory of depth $d$ is adapted from \cite{Pollock_newton} and described in Algorithm \ref{alg:AndersonAcc_depthD}.
\begin{algorithm}[htbp]
\caption{AA$(d)$-Inner Iterations}
\label{alg:AndersonAcc_depthD}
\begin{algorithmic}[1]
\REQUIRE $\hat{w}\in\mathcal{F}$, fixed Jacobian $G^{\top}=\nabla g(\hat{w})^{\top}P_y$. Let $\bar{w}$ be the solution of \eqref{eq:LP} at $\hat{w}$. $\Delta>0$,  $\sigma_{\mathrm{inner}}\in (0,\,10^{-5})$;
\ENSURE $\tilde{w}$
\STATE $w_0 \leftarrow\hat{w}$, $w_1\leftarrow \bar{w}$, $r_1\leftarrow w_1 - w_0$
\FOR{$l=1,\,2, \ldots$}
\STATE Set $m = \min\{l,\,d\}$
\IF{$h(w_l) \leq \sigma_{\mathrm{inner}}$ and $\Vert \bar{w}-w_l\Vert / \Vert \bar{w}-\hat{w}\Vert < 1/2$}
\STATE $\tilde{w}\leftarrow w_l$
\STATE \textbf{STOP}
\ENDIF
\STATE Solve $\mathrm{PLP}(w_l,\,\hat{w},\,\Delta)$
\IF{iterates $w_l$ are diverging}
\STATE \textbf{STOP}
\ENDIF
\STATE Define $r_{l+1} = w_{\mathrm{PLP}}^*(w_l;\,\hat{w},\Delta)-w_l$
\STATE Set $F_l = ((r_{l+1}-r_l) \ldots (r_{l-m+2}-r_{l-m+1})) \in \mathbb{R}^{(n_w,\,m)}$ and $E_l = ((w_l-w_{l-1}) \ldots (w_{l-m+1}-w_{l-m})) \in \mathbb{R}^{(n_w,\,m)}$
\STATE Compute $\gamma_{l+1} = \mathrm{argmin}_{\gamma\in\mathbb{R}^m} \Vert r_{l+1} - F_l \gamma\Vert$
\STATE Set $w_{l+1} = \Pi_{\mathcal{B}_{\infty}(\hat{w}, \Delta)}(w_l + r_{l+1} - (E_l + F_l)\gamma_{l+1})$
\ENDFOR
\end{algorithmic}
\end{algorithm}
In the case where the memory $d=1$, calculating $\gamma_{l+1}$ simplifies to
\begin{align}
    \gamma_{l+1} = \frac{\langle r_{l+1}, r_{l+1}-r_l \rangle}{\Vert r_{l+1} - r_l\Vert^2}.
\end{align}
Algorithms \ref{alg:feas_iter} and \ref{alg:AndersonAcc_depthD} differ in their step updates. In AA($d$), the new iterate is computed as an affine combination of the previous steps in memory. The projection operator $\Pi_{\mathcal{B}_{\infty}(\hat{w}, \Delta)}$ ensures that all Anderson iterates stay within the trust-region. The operator clips the components of the vector $w\in\mathbb{R}^{n_w}$ to $\mathcal{B}_{\infty}(\hat{w}, \Delta)$.\\
\cite{Toth2015} show that the contraction rate of their Anderson accelerated method is at least as good as the one of the fixed point iteration. A proof that the contraction rate can be improved is given by \cite{Pollock2020a}. The remainder of this section will focus on proving the convergence of algorithm AA($d$).
\begin{lemma}
\label{lemma:BoundedW}
Let Assumption \ref{as:invertiblePLP} hold, then for the iterates $w_l$ and steps $r_{l+1}$ produced by Algorithm \ref{alg:AndersonAcc_depthD} there exist $L, \Delta >0$ such that
\begin{align}
    \Vert w_l + r_{l+1} - w^*\Vert \leq L\Delta \Vert w_l - w^*\Vert.
\end{align}
\end{lemma}
\begin{proof}
The proof follows the same argument as the proof of Theorem \ref{theo:localContractionJosephyNewton} given in \cite{Kiessling2022}. Note that $w_l,\,w_l+r_{l+1}\in\mathcal{B}_{\infty}(\hat{w}, \Delta)$. Due to Assumption \ref{as:invertiblePLP}, it holds that
\begin{align*}
    \Vert w_{l} + r_{l+1} - w^* \Vert &=
    \Vert w_{\mathrm{PLP}}^*(w_{l};\,\hat{w},\Delta) - w_{\mathrm{PLP}}^*(w^*;\,\hat{w},\Delta) \Vert\\
&\leq L_1 \Vert \delta_l - \delta^*\Vert.
\end{align*}
Using the fundamental theorem of calculus, Lipschitz continuity of the Jacobian of $g(P_y \cdot)$ with Lipschitz constant $\tilde{L}>0$ and $w_l \in \mathcal{B}_{\infty}(\hat{w},\,\Delta)$, for $\delta_l$ it can be shown that
\begin{align*}
&\Vert\delta_l - \delta^*\Vert = \Vert g(P_y w_l)-g(P_y w^*)-G^{\top}(w_l-w^*)\Vert\\
&\leq \int_0^1 \tilde{L} \Vert w^* + t(w_l-w^*) -\hat{w}\Vert \mathrm{d}t\:\Vert w_l-w^*\Vert\\
&\leq \tilde{L}\,\Delta\,\Vert w_l-w^*\Vert.
\end{align*}
With $L:=L_1\, \tilde{L}$, we get $\Vert w_{l+1} - w^* \Vert \leq L \Delta\,  \Vert w_{l}-w^*\Vert$.
In particular $L, \Delta$ are the same as in Theorem \ref{theo:localContractionJosephyNewton}.
\end{proof}
In order to show convergence, we need a bound on the calculated least squares solutions $\gamma_{l+1}$.
\begin{assumption}[Boundedness of $\gamma_l$]
\label{as:boundedGamma}
The $\gamma_{l}$ defined in Algorithm \ref{alg:AndersonAcc_depthD} are uniformly bounded, i.e., there exists a $D > 0$ such that $\Vert \gamma_l \Vert_{\infty} \leq D$, $\forall l \in \{2,\,3,\ldots\}.$
\end{assumption}

\begin{theorem}[Convergence of AA($d$)]
Let Assumption \ref{as:invertiblePLP} and \ref{as:boundedGamma} hold, then the iterates produced by Algorithm \ref{alg:AndersonAcc_depthD} with depth $d$ converge linearly, i.e.,
\begin{align}
\label{eq:convAnderson}
    \Vert w_{l+1}-w^* \Vert \leq L\Delta \bar{D} \sum_{j=l-d}^{l}\Vert w_j -w^* \Vert,
\end{align}
where $\bar{D}:=\max\{2 D,\,1+ D\}$.
\end{theorem}
\begin{proof}
The proof is similar to the convergence proof in \cite{Pollock_newton}. For simplicity of presentation, the iteration index of $\gamma_{l+1}$ will be dropped in this proof. Expanding the update step in Algoritm \ref{alg:AndersonAcc_depthD} yields
\begin{align*}
\check{w}_{l+1}= \,&w_l + r_{l+1}\\
&- \sum_{j=l-d+1}^l ((w_j -w_{j-1}) + (r_{j+1}-r_j))\gamma^{l-j+1}\\
= \,&(w_l + r_{l+1})(1-\gamma^1)\\
&+ \sum_{j=l-d+1}^{l-1} (w_j + r_{j+1})(\gamma^{l-j} - \gamma^{l-j+1})\\
&+ (w_{l-d} + r_{l-d+1})\gamma^d.
\end{align*}
Using
\begin{align*}
    w^* = w^*((1-\gamma^1) + (\gamma^1 - \gamma^2) + \ldots + (\gamma^{d-1} - \gamma^d) + \gamma^d)
\end{align*}
we obtain
\begin{align*}
\check{w}_{l+1}-w^*= \,&(w_l + r_{l+1} -w^*)(1-\gamma^1)\\
&+ \sum_{j=l-d+1}^{l-1} (w_j + r_{j+1} -w^*)(\gamma^{l-j} - \gamma^{l-j+1})\\
&+ (w_{l-d} + r_{l-d+1} -w^*)\gamma^d.
\end{align*}
Taking the norm $\Vert \cdot \Vert$, using the triangle inequality, recalling that $w_{l+1} = \Pi_{\mathcal{B}_{\infty}}(\check{w}_{l+1})$ with $\Vert w_{l+1} -w^*\Vert \leq \Vert \check{w}_{l+1} - w^*\Vert$ and applying
Lemma \ref{lemma:BoundedW} yields
\begin{align*}
\Vert w_{l+1}-w^*\Vert \leq \,&L\Delta\Vert w_l -w^*\Vert(1-\gamma^1)\\
&+ L\Delta\sum_{j=l-d+1}^{l-1} \Vert w_j -w^*\Vert(\gamma^{l-j} - \gamma^{l-j+1})\\
&+ L\Delta\Vert w_{l-d} -w^*\Vert\gamma^d.
\end{align*}
Due to Assumption \ref{as:boundedGamma}, choosing $\bar{D}:=\max\{2 D,\,1+ D\}\geq\max\{2 \Vert \gamma\Vert_{\infty},\,1+ \Vert \gamma\Vert_{\infty}\}$ yields \eqref{eq:convAnderson}.
\end{proof}
\begin{remark}
The proof of the projection ratio condition for the inner iterations of FSLP is directly applicable to AA($d$). Therefore, global convergence of the outer iterations is also guaranteed if Anderson acceleration is used in the inner iterations.
\end{remark}
If FSLP is combined with the Anderson acceleration method, it will be denoted by AA($d$)-FSLP.

\section{Simulation Example}
\label{section_example}
In this section, the performance of the AA($d$)-FSLP algorithm is demonstrated on a time-optimal P2P motion planning problem of a parallel SCARA robot. Firstly, implementation details are given and the optimal control problem is defined. Secondly, the improved contraction rate in the feasibility phase of AA($d$)-FSLP in comparison to FSLP is demonstrated. Finally, the convergence improvement of the outer iterations is illustrated on a fully determined and an under determined SCARA problem.

\subsection{Implementation}
We use the Python interface of the open source software \texttt{CasADi} \citep{Andersson2019} to model the optimization problem and for a prototypical implementation of the AA($d$)-FSLP solver. As LP solver, we use the dual simplex algorithm of \texttt{CPLEX} version 12.8 \citep{cplex2017v12}, which can be called from \texttt{CasADi}. 
 The parameters of the outer algorithm, in Algorithm \ref{alg:feas_iter}, and \ref{alg:AndersonAcc_depthD} are chosen as in \cite{Kiessling2022}. The simulations were carried out on an Intel Core i7-10810U CPU.

\subsection{Formulation of the P2P motion problem}
AA($d$)-FSLP is tested on a time-optimal P2P motion problem of the following form:
\begin{subequations}
\label{OCP}
\begin{align}
\min_{\substack{x_0,\ldots,\,x_N \\ u_0,\ldots,\,u_{N-1} \\ s_0,\,s_N,\,T}} &T+ \mu_0^{\top} s_0 +\mu_N^{\top} s_N\\
\mathrm{s.t.}\quad &-s_0 \leq x_0-\overline{x}_0 \leq s_0,\\
\quad &x_{k+1}=f(x_k,\,u_k,\,\tfrac{T}{N}),&&\hspace{-6.5mm}k=0,\dots,\,N-1, \label{eq:ms_constraint}\\
\quad &u_k \in \mathbb{U}_k,&&\hspace{-6.5mm}k=0,\dots,\,N-1, \label{eq:input_constraint}\\
\quad &x_{k} \in \mathbb{X}_{k},&&\hspace{-6.5mm}k=0,\dots,\,N, \label{eq:state_constraint}\\
\quad &e(x_k,\,u_k)\leq 0,&&\hspace{-6.5mm}k=0,\dots,\,N-1, \label{eq:stage_constraint}\\
\quad &-s_N \leq x_N-\overline{x}_N \leq s_N.
\end{align}
\end{subequations}
where $x_k\in\mathbb{R}^{n_x},\,  u_k\in\mathbb{R}^{n_u},\,s_0,\,s_N\in\mathbb{R}^{n_x}$ denote the state, control, and slack variables for horizon length $N\in\mathbb{N}$. The time horizon is given by $T\in\mathbb{R}_{>0}$ and the multiple shooting time interval size is given by $h:=\frac{T}{N}$. We denote the start and end points by $\bar{x}_0,\,\bar{x}_N\in\mathbb{R}^{n_x}$. Let $f\colon \mathbb{R}^{n_x}\times\mathbb{R}^{n_u} \to \mathbb{R}^{n_x}$ and $e\colon\mathbb{R}^{n_x}\times\mathbb{R}^{n_u} \to \mathbb{R}^{n_e}$ denote the system dynamics and the stage constraints, respectively. Additionally, let $\mu_0,\,\mu_N\in\mathbb{R}^{n_s}_{> 0}$ denote penalty parameters and let $\mathbb{U}_k,\,\mathbb{X}_k$ denote convex polytopes. Note that the penalty terms in the objective function are needed such that AA($d$)-FSLP can be started from a feasible point. All elements of \eqref{OCP} are further discussed in the following section.

\subsection{SCARA robot motion planning problem}
The parallel SCARA robot, shown in Fig. \ref{fig:SCARA}, is a two-$\mathrm{DOF}$ five-bar planar parallel manipulator composed by two 2-link arms whose end-effectors are constrained to be attached to a revolute joint $j_5$, forming a loop closure constraint. 
\begin{figure}[htbp]  
\subfigure{\includegraphics[width = 0.50\linewidth]{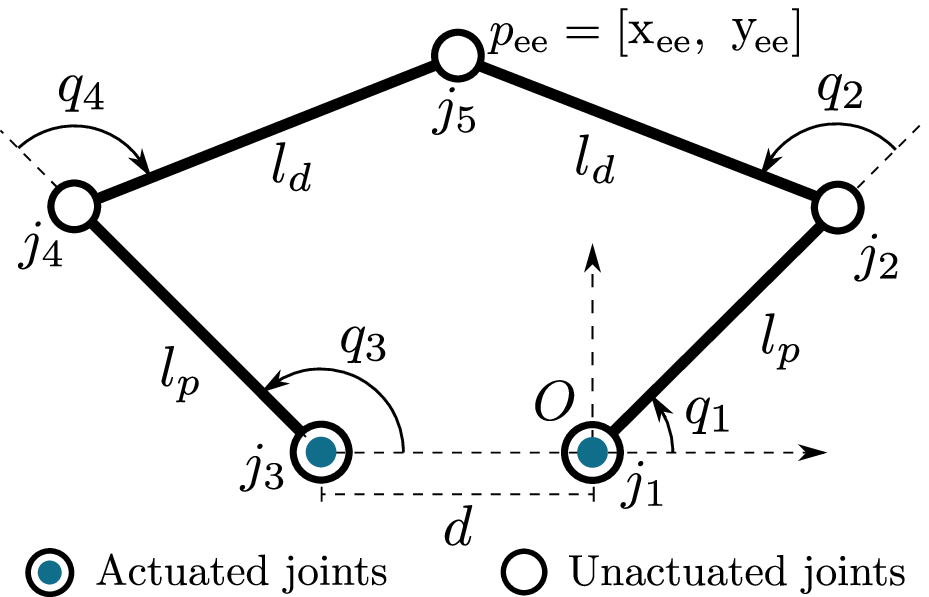}} 
\subfigure{\includegraphics[width = 0.39\linewidth]{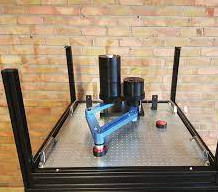}}  
\caption{Schematic (left) and photograph (right) of the parallel SCARA robot. $j_1,\ j_3$ denote the actuated joints, $j_2,\ j_4,\ j_5$ the unactuated joints, $q_i$ the angular position of joint $j_i$, $d$ the distance between the actuated joints, $l_p$ the length of the links proximal to the actuated joints, $l_d$ the length of the links distal to the actuated joints, and $p_{\mathrm{ee}}$ the position of the end-effector.}
\label{fig:SCARA}
\end{figure}
Let $q := \begin{bmatrix} q_1, & q_3 \end{bmatrix}^\top \in \mathbb{R}^2$ denote the independent coordinates, and $\bar{q} := \sigma(q) = \begin{bmatrix} q_1, & q_2(q), & q_3, & q_4(q) \end{bmatrix}^\top \in \mathbb{R}^4$ the generalized coordinates of the system.\\
Considering the  state vector $x := \begin{bmatrix} q^{\top}, & \dot{q}^{\top} \end{bmatrix}^\top \in \mathbb{R}^4$ and the control (torque) input $u := \tau \in \mathbb{R}^2$, the dynamics of the parallel SCARA robot are expressed by the ordinary differential equation (ODE) \citep{Cheng_2011}
\begin{equation} \label{eq:ode}
\dot{x} = f_{\mathrm{ode}}(x,u) := \begin{bmatrix} \dot{q} \\  {M}^{-1}(\bar{q})(\tau  - {F}(\bar{q},\dot{\bar{q}}))\end{bmatrix},
\end{equation}
where ${M}$ is the inertia matrix and ${F}$ is the vector of Coriolis and centrifugal effects.
Function $f_{\mathrm{ode}}$ is generated by using the rigid-body dynamics library \texttt{Robotran} \citep[Sec. 2.2]{robotran} and takes into account (i) the loop closure constraint in the acceleration level by using the coordinate partitioning  technique \citep{Wehage_1982},
and (ii) the dependency of the generalized coordinates $\bar{q} := \sigma(q)$ on $q$ \citep{Ouyang_2004, Cheng_2011}.\\
The constraints of the time-optimal P2P motion problem \eqref{OCP} for the parallel SCARA robot are defined as follows. Function $f$ in \eqref{eq:ms_constraint} is generated by discretizing $f_{\mathrm{ode}}$ in \eqref{eq:ode} by means of an explicit $4^{\mathrm{th}}$-order Runge-Kutta integrator. The convex polytopes $\mathbb{U}_k$ and $\mathbb{X}_k$, defined by lower and upper bounds imposed to $\tau$ and $\bar{q}$, respectively, set box constraints on the inputs $u_k$ and states $x_k$ in \eqref{eq:input_constraint} and \eqref{eq:state_constraint}. Two instances of the stage constraint \eqref{eq:stage_constraint} are included, namely
\begin{align}
  & \Vert\dot{p}_{\mathrm{ee}}\Vert^2 \leq V_{\mathrm{max}}^{2}, \label{eq:max_vel_constraint}
  \\
  & \begin{cases}
    \mathbf{n}_{\mathrm{a}}^\top p_{\mathrm{ee}} + \mathbf{n}_{\mathrm{b}} + r_{\mathrm{safe}} \leq 0, & \\
    \forall \upsilon \in \mathcal{V}_{\mathrm{obs}},\ \mathbf{n}_{\mathrm{a}}^\top \upsilon + \mathbf{n}_{\mathrm{b}} \geq 0, & \\
    \Vert\mathbf{n}\Vert_{\infty} \leq 1, & 
  \end{cases} \label{eq:hyperplane_constraints}
\end{align}
where \eqref{eq:max_vel_constraint} sets an upper bound $V_{\mathrm{max}} \in \mathbb{R}_{>0}$ to the squared $\ell_2$-norm of the velocity of the end-effector, and \eqref{eq:hyperplane_constraints} sets collision avoidance constraints based on separating hyperplanes defined by the parameter $\mathbf{n} := \begin{bmatrix}\mathbf{n}_{\mathrm{a}}^\top & \mathbf{n}_{\mathrm{b}} \end{bmatrix}^\top \in \mathbb{R}^3$, a safety margin $r_{\mathrm{safe}} \in \mathbb{R}_{\geq 0}$, and an obstacle defined by a set of vertices $\mathcal{V}_{\mathrm{obs}}$.

\begin{remark}
Please note that the inclusion of constraint \eqref{eq:max_vel_constraint} prevents the system from being fully determined, as explained in Section \ref{sec:determined_case}.
\end{remark}
The parameters in the simulation example are taken from the SCARA robot seen in Fig. \ref{fig:SCARA} on the right. It holds that $l_d= 0.28 \mathrm{[m]}$, $l_p= 0.18 \mathrm{[\mathrm{m}]}$. The start point in cartesian space is given by $\bar{p}_{0}=[0.0,\, 0.115]^{\top}[m]$ and  the end point by $\bar{p}_{N}=[0.0,\, 0.405]^{\top}[\mathrm{m}]$ transforming to $\bar{x}_0$ and $\bar{x}_N$ through inverse kinematics. The obstacle is a square whose vertices are defined by its center of gravity at $[0.0,\,0.2]^{\top}[\mathrm{m}]$ and its side length of $0.02 \mathrm{[m]}$. The lower bounds on the joint angles are given by $\bar{q}_{\mathrm{min}} = [-\frac{\pi}{6},\,-\frac{11\pi}{12},\,\frac{\pi}{6},\,-\frac{11\pi}{12}]^{\top} [\mathrm{rad}]$ and the upper bounds by $\bar{q}_{\mathrm{max}} = [\frac{5\pi}{6},\,\frac{11\pi}{12},\,\frac{7\pi}{6},\,\frac{11\pi}{12}]^{\top} [\mathrm{rad}]$. For the bounds on the torques it holds $\tau_{\mathrm{max}} = -\tau_{\mathrm{min}} = [5.0,\,5.0]^{\top}[\mathrm{Nm}]$. The maximum velocity is set equal to $V_{\mathrm{max}}=2.0[\frac{\mathrm{m}}{\mathrm{s}}]$.\\
For an initialization at the point $[0.05,\,0.12]^{\top}[\mathrm{m}]$ in cartesian space, the time-optimal motion trajectory obtained with the FSLP algoritm and initialized with a constant control input of $[0.05, -0.035]^{\top}[\mathrm{Nm]}$ over a time horizon of $0.7\mathrm{[s]}$ is shown in Fig. \ref{fig:SCARA_trajectory}.\\
\begin{figure}[thpb]
\includegraphics[width=0.48\textwidth]{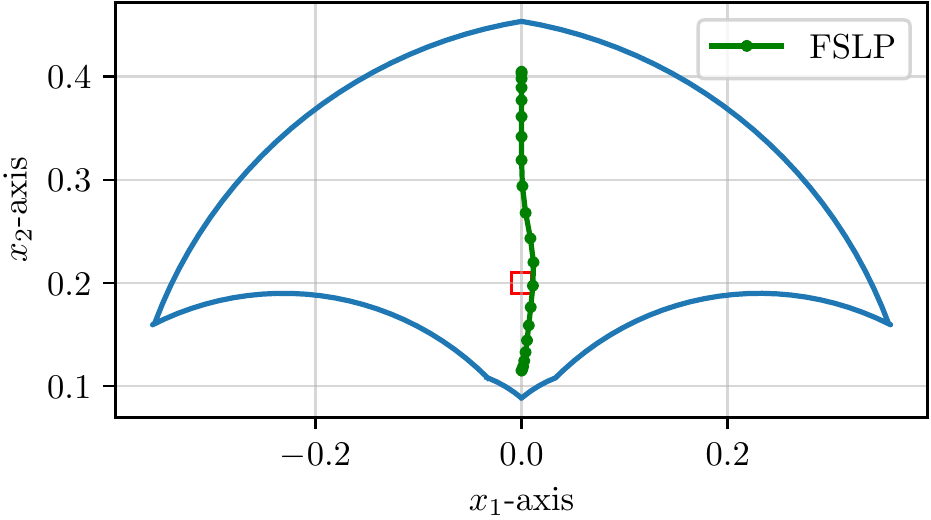}
\caption{Optimal trajectory of FSLP for the time-optimal P2P motion problem. The blue contours denote the workspace of the SCARA robot.}
\label{fig:SCARA_trajectory}
\end{figure}
The AA($d$)-FSLP algorithm is tested on a set of perturbed SCARA robot optimization problems. Small, uniformly distributed perturbations around zero are added to the starting point and to the end point of the P2P motion problem. In total 100 different optimization problems are considered. In the following, the set of test problems will be denoted as \textit{SCARA test set}.

\subsection{Observation of convergence rate in inner iterations}
It is observed that AA($d$) improves the convergence rate of the inner iterations. For an initial trust-region radius of $\Delta = 0.25$ at the initialization of the unperturbed SCARA problem, Fig. \ref{fig:contraction_steps} illustrates the convergence of the inner iterates of FSLP compared to AA($d$)-FSLP with $d\in\{1,\,5,\,15\}$.
\begin{figure}[thpb]
\includegraphics[width=0.48\textwidth]{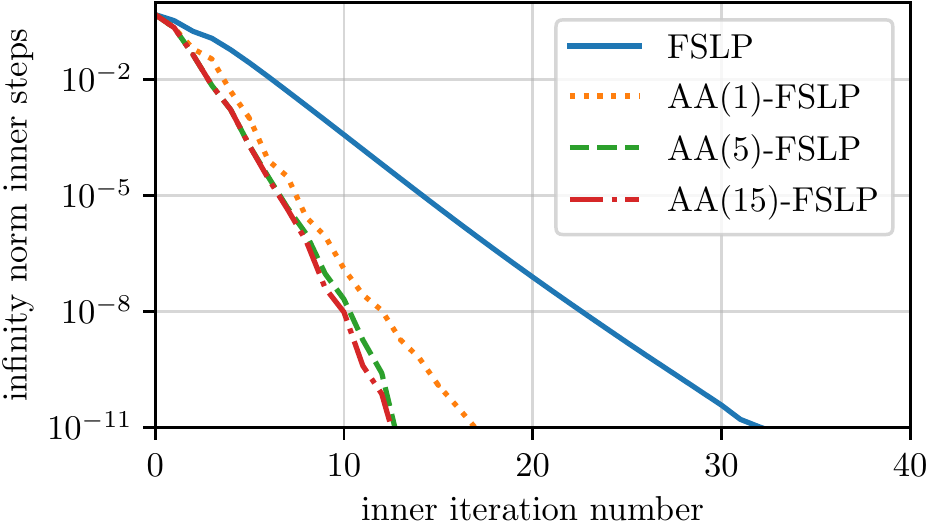}
\caption{Comparison of the convergence rates of FSLP and AA($d$)-FSLP for $d\in\{1,\,5,\,15\}$.}
\label{fig:contraction_steps}
\end{figure}
The plot shows that the contraction rate is improved by higher memory depths $d$. For the demonstrated problem, we also see that the contraction rate cannot be significantly improved by high values of $d$ due to the limited number of iterations needed for convergence. Since AA($d$)-FSLP takes fewer iterations to converge in the inner iterations, the constraints need to be evaluated fewer times. On the SCARA test set, the AA($d$)-FSLP can significantly decrease the number of constraint evaluations. The mean values of the constraint evaluations $\bar{n}_{\mathrm{con}}$ are shown in Table \ref{tab:Mean_eval_g}.
\begin{table}[]
\caption{Means of number of constraint evaluations, outer iterations, and wall time on SCARA test set.}
\label{tab:Mean_eval_g}
\centering
\begin{tabular}{l c c c c}
\toprule
                         & FSLP   & AA(1)-FSLP & AA(5)-FSLP & AA(15)-FSLP \\
\midrule
$\bar{n}_{\mathrm{con}}$ & 448.84 & 274.69     & 196.57     & 193.98 \\   
$\bar{n}_{\mathrm{iter}}$ & 49.68 & 32.31     & 27.62     & 27.70 \\ 
$\bar{t}_{\mathrm{wall}} [\mathrm{s}]$ & 5.01 & 2.914     & 2.24     & 2.34 \\ 
\bottomrule
\end{tabular}
\end{table}

\subsection{Fully Determined Case and Under Determined Case} \label{sec:determined_case}
\cite{Kiessling2022} exploit the property when an NLP is fully determined to achieve locally quadratic convergence of the FSLP algorithm. An NLP is fully determined if the following condition is satisfied \citep{Messerer2021}:
\begin{definition}[Fully Determined Solution]
Let the NLP \eqref{eq:nlp} be given and let the set of active constraints at a feasible point $w\in\mathbb{R}^{n_w}$ be $\mathcal{A}(w)$.
A solution $w^*\in\mathbb{R}^{n_w}$ of \eqref{eq:nlp} is \textbf{fully determined} by the constraints if $n_g+|\mathcal{A}(w^*)|=n_w$ and LICQ holds at $w^*$. If $n_g+|\mathcal{A}(w^*)|<n_w$, the solution is called \textbf{under determined}.
\end{definition}
If the optimal solution of the optimization problem \eqref{eq:nlp} is fully determined, the solution lies in a vertex of the active constraints, which reduces the optimization problem to finding a feasible point. In this case, no Hessian information is needed and SLP converges quadratically \citep{Messerer2021}. For time-optimal control problems, it is easy to construct counter-examples where the solution is not fully determined. In the case of the SCARA robot, we impose constraints on the speed of the end effector that are reasonable for the safe operation of the robot. These constraints restrict the speed within an $l_2$-ball, which does not have vertices. This prevents the problem from being fully determined. 
\begin{figure}[thpb]
\includegraphics[width=0.48\textwidth]{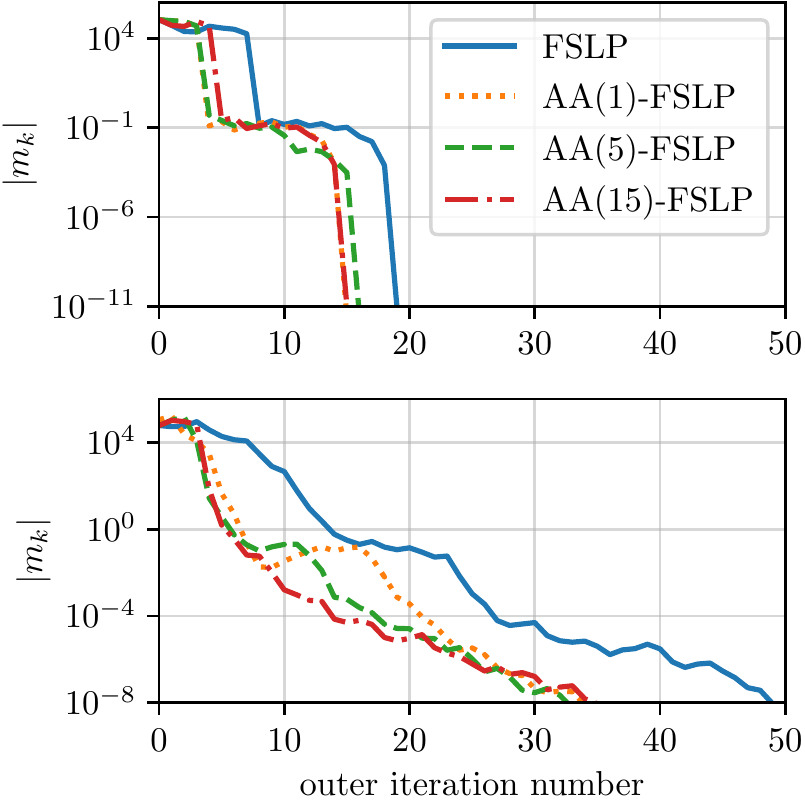}
\caption{Comparison of the convergence with respect to the absolute value of the convergence metric $m_k$ of FSLP and AA($d$)-FSLP for $d\in\{1,\,5,\,15\}$. The top figure shows the convergence in the case of the fully determined system and the lower figure for the under determined system.}
\label{fig:contraction_outer}
\end{figure}
In Fig. \ref{fig:contraction_outer}, the difference in convergence for a fully determined and an under determined system is shown. The fully determined problem as seen in the upper figure of Fig. \ref{fig:contraction_outer} can be reached by removing constraint \eqref{eq:max_vel_constraint} from the optimization problem \eqref{OCP}. It is seen that the difference in the number of outer iterations taken until convergence is small in the fully determined problem due to the local quadratic convergence. But in the case of an under determined problem, AA($d$) can achieve a significant improvement in the convergence of outer iterates towards an optimal point. This is due to the fact that in a case where the termination heuristic of FSLP would already terminate the inner iterations since the convergence rate is above a given threshold, the accelerated methods often fulfill the desired convergence rates. Instead of decreasing the trust-region radius, this gives longer steps that yield faster convergence. In \cite{Tenny2004}, where a feasible sequential quadratic programming method is investigated the same effect is observed.\\
Table \ref{tab:Mean_eval_g} reports the mean of the number of iterations $\bar{n}_{\mathrm{iter}}$ of FSLP and AA($d$)-FSLP. On average, the number of outer iterations is reduced by AA($d$). This is also represented in the average wall times needed for solving the problem. A speed-up of at least $40$\% is achieved by using AA($d$).

\section{Conclusion}
\label{section_conclusion}
This paper proposed an acceleration method for the feasibility iteration phase of FSLP. The reduction of outer iterations and constraint evaluations was successfully demonstrated on a set of time-optimal point-to-point motion problems of a SCARA robot. As a consequence, the overall solution time was decreased. Especially in the case of an under determined system, the outer iteration number of FSLP was significantly reduced. Moreover, it was demonstrated that AA($d$) can reduce the number of feasibility iterates which can be advantageous in real-time applications such as MPC, MHE, and ILC. A more detailed investigation of Anderson acceleration in zero-order optimization algorithms is the subject of current research.


{\small
\bibliography{bibliography}             
}
                                                   







\end{document}